\documentclass[10pt]{article}

\usepackage{amsthm}
\usepackage{amsmath}
\usepackage{amssymb}
\usepackage{xypic}
\usepackage{hyperref}
\usepackage[nottoc,numbib]{tocbibind}

\newtheorem{thm}{Theorem}[section]

\newtheorem{emp}[thm]{Example}
\newtheorem{lem}[thm]{Lemma}

\newtheorem{defn}[thm]{Definition}

\newtheorem{cor}[thm]{Corollary}
\newtheorem{thm_intro}[thm]{Theorem}

\begin{document}
\title{On the Tannakian group scheme of integrable connections}
\author{Xiaodong Yi\footnote{yixd97@outlook.com}}
\date{}
\maketitle
\begin{abstract}
We give a counterexample showing that, over a field of characteristic $0$, the Künneth formula does not hold for the Tannakian group scheme of all integrable connections.
\end{abstract}
\noindent Keywords: integrable connection, Tannakian category
\section{Introduction}
The Künneth formula is important in algebraic topology, as it allows us to obtain the fundamental group of a finite product of spaces from the fundamental groups of all factors. This article is devoted to the study of an algebro-geometric fundamental group built from integrable algebraic connections and the associated Tannakian category. Let $k$ be an algebraically closed field of characteristic $0$. For any smooth connected scheme $X$ over $k$, the category of all integrable connections on $X$ is Tannakian, and is neutralized provided with a base point $x\in X(k)$. We use $\pi_{1}^{diff}(X,x)$ to denote the corresponding Tannakian group scheme over $k$. Suppose we have two such pointed schemes $(X,x)$ and $(Y,y)$. There is a natural morphism
\begin{equation}\label{kunneth_intro}\pi_{1}^{diff}(X\times_{k} Y, x\times y)\rightarrow \pi_{1}^{diff}(X,x)\times_{k}\pi_{1}^{diff}(Y,y).\end{equation} \par
It is not difficult to see that the morphism \eqref{kunneth_intro} is faithfully flat (see Corollary \ref{cir} below), and more precisely we would like to ask if the morphism \eqref{kunneth_intro} is an isomorphism, which is referred as the Künneth formula for $\pi_{1}^{diff}$. For example, in case $k=\mathbb{C}$ is the field of complex numbers, and $X,Y$ are proper over $k$, the Künneth formula can be reduced to the usual Künneth formula for topological spaces, by GAGA and the Riemann-Hilbert correspondence. More generally, we can also obtain a Künneth formula for non-proper $X$ and $Y$, at the cost of working with the Tannakian group scheme of regular singular connections, which is a quotient of $\pi_{1}^{diff}$, by applying Deligne's Riemann-Hilbert correspondence (\cite{deligne1970equations}). To go beyond the usual  Künneth formula for topological spaces, we would like to allow irregular singularities for integrable connections. A first progress in this direction was made in \cite{zhang2014homotopy}, where Zhang proved that, to ensure that the morphism \eqref{kunneth_intro} is an isomorphism, it suffices to assume one of the factors $X$ and $Y$ to be proper over $k$. The same result was later obtained in \cite{dos2015homotopy} by dos Santos, and his argument worked for $\mathcal{D}$-modules in characteristic $p>0$ as well. \par
Our result is a complement to the results from Zhang and from dos Santos, saying that without any properness assumption, the Künneth formula does not hold even in the simplest case (so the properness assumption is necessary, at least partially):
\begin{thm_intro}[Corollary \ref{cir}, Theorem \ref{kun_diff}]
Let $X=Y$ be the affine line $\mathbb{A}_{k}^{1}$. The natural morphism 
\[ \pi_{1}^{diff}(X\times_{k} Y, x\times y)\rightarrow \pi_{1}^{diff}(X,x)\times_{k}\pi_{1}^{diff}(Y,y)\]
is faithfully flat, but not a closed immersion. 
\end{thm_intro}
 
\subsection*{\normalsize Conventions and Notations}
\begin{description}
\item[Base field:]  throughout the article, we use $k$ to denote an algebraically closed field of characteristic $0$. All schemes are assumed to be separated and of finite type over $k$.
\item[Connection:]for a smooth scheme over $k$, an integrable connection is a pair $(\mathcal{E},\nabla)$, where $\mathcal{E}$ is a coherent $\mathcal{O}_{X}$-module and $\nabla$ is a $k$-linear morphism
\[\nabla: \mathcal{E}\rightarrow \mathcal{E}\otimes\Omega_{X/k}^{1}\]
verifying the Leibniz rule. A connection with $\nabla^{2}=0$ is said to be integrable. We use $\mathrm{Conn}(X)$ to denote the category of integrable connections on $X$. We only work with integrable connections so the adjective ``integrable'' is occasionally omitted. When $X$ is connected, smooth over $k$ and equipped with  a base point $x\in X(k)$, the category $\mathrm{Conn}(X)$ is a neutral Tannakian category, and the fiber functor is given by $(\mathcal{E},\nabla)\mapsto \mathcal{E}_{x}$. The corresponding Tannakian group scheme is denoted by $\pi_{1}^{diff}(X,x)$.
\item [Tannakian subcategory:] let $\mathcal{C}$ be a Tannakian category, and $T$ be a set of objects in $\mathcal{C}$. Assume the unit object is contained in $T$. We use $\langle T \rangle$ to denote the full Tannakian subcategory generated by $T$, i.e, the smallest full subcategory of $\mathcal{C}$ such that 
\begin{enumerate}
\item It contains all objects in $T$.
\item It is closed under taking subquotients, duals, and tensor products in $\mathcal{C}$.
\end{enumerate}
\end{description}

\section{\normalsize Some Tannakian preliminaries}
We have the following simple lemma:
\begin{lem}
\label{subquotient}
Let $G$ and $H$ be affine group schemes over $k$. Any finite dimensional $G\times_{k} H$-representation $U$ is a subquotient of a $G\times_{k} H$-representation of the form $V\otimes W$, where $V$ (resp. $W$) is a finite dimensional $G$ (resp. $H$)-representation. 
\end{lem}
\begin{proof}
The space of $H$-linear endomorphisms $\mathrm{End}_{H}(U)$ is a finite dimensional $G$-representation by setting $(g\cdot\rho)(u)=g\cdot \rho(u)$, for  $g\in G,\rho\in\mathrm{End}_{H}(U),u\in U$. We conclude from the $G\times_{k} H$-equivariant surjection
\[\mathrm{End}_{H}(U)\otimes U \twoheadrightarrow U, (\rho,u)\mapsto \rho(u),\]
where the right tensor factor $U$ is considered only as an $H$-representation by forgetting the $G$-action on it.
\end{proof}
To simplify the notation, we write $\mathcal{E}$ for a connection $(\mathcal{E},\nabla)$ by omitting the connection map $\nabla$. \par
We fix some notations. Let $X$ and $Y$ be connected and smooth over $k$.  Let $x\in X(k)$ (resp. $y\in Y(k)$) be the chosen base point, and $i_{x}$ (resp $i_{y}$) be the inclusion $i_{x}:Y=\{x\}\times Y\hookrightarrow X\times_{k} Y$ (resp. $i_{y}:X=X\times\{y\}\hookrightarrow X\times_{k} Y$). Let $p_{X}$ (resp. $p_{Y}$) be the projection $X\times_{k} Y\rightarrow X$ (resp. $X\times_{k} Y\rightarrow Y$).  Let $\mathrm{Conn}(X)\boxtimes\mathrm{Conn}(Y)$ be $\langle\mathcal{F}\boxtimes\mathcal{G},\mathcal{F}\in\mathrm{Conn}(X),\mathcal{G}\in\mathrm{Conn}(Y)\rangle$, where  $\mathcal{F}\boxtimes\mathcal{G}$ is defined to be $p_{X}^{*}\mathcal{F}\otimes p_{Y}^{*}\mathcal{G}$. We write $\mathrm{Aut}(\mathrm{Conn}(X)\boxtimes\mathrm{Conn}(Y))$ for the Tannakian group scheme associated to $\mathrm{Conn}(X)\boxtimes\mathrm{Conn}(Y)$.\par
 The pull-back $p_{X}^{*}$ defines a tensor functor from $\mathrm{Conn}(X)$ to $\mathrm{Conn}(X\times_{k} Y)$ and thus a morphism $\pi_{1}^{diff}(X\times_{k} Y,x\times y)\rightarrow \pi_{1}^{diff}(X,x)$.  By the same reason from the pull-back $p_{Y}^{*}$ we have a morphism $\pi_{1}^{diff}(X\times_{k} Y,x\times y)\rightarrow\pi_{1}^{diff}(Y,y)$. Combining these two morphisms and now we have 
\begin{equation}\label{structure}\pi_{1}^{diff}(X\times_{k} Y,x\times y)\rightarrow \pi_{1}^{diff}(X,x)\times_{k}\pi_{1}^{diff}(Y,y).\end{equation}
\begin{lem}
\label{tannakian_lemma}
We have the following assertions:
\begin{enumerate}
\item Any object in $\mathrm{Conn}(X)\boxtimes\mathrm{Conn}(Y)$ can be written as $\mathcal{E}_{1}/\mathcal{E}_{2}$, with $\mathcal{E}_{2}\hookrightarrow \mathcal{E}_{1}\hookrightarrow \mathcal{F}\boxtimes\mathcal{G}$ for some $\mathcal{F}\in\mathrm{Conn}(X)$ and $\mathcal{G}\in\mathrm{Conn}(Y)$. 
\item The morphism \eqref{structure} factors through $\mathrm{Aut}(\mathrm{Conn}(X)\boxtimes\mathrm{Conn}(Y))$ and leads to a factorization:
\begin{equation}\label{fac}\pi_{1}^{diff}(X\times_{k} Y,x\times y) \twoheadrightarrow \mathrm{Aut}(\mathrm{Conn}(X)\boxtimes\mathrm{Conn}(Y)) \xrightarrow{\sim} \pi_{1}^{diff}(X,x)\times_{k} \pi_{1}^{diff}(Y,y). \end{equation}
\end{enumerate}
\end{lem}
\begin{proof}
We prove the first part of the lemma. We  prove that the full subcategory of $\mathrm{Conn}(X\times_{k} Y)$ consisting of objects of the form $\mathcal{E}_{1}/\mathcal{E}_{2}$ is closed under taking subquotients, tensor products and duals. Therefore it must coincide with $\mathrm{Conn}(X)\boxtimes\mathrm{Conn}(Y)$. \par
\begin{enumerate}

\item It is clear that objects of the form $\mathcal{E}_{1}/\mathcal{E}_{2}$ form a full subcategory of $\mathrm{Conn}(X\times_{k} Y)$, which is closed under taking subobjects, quotients, and therefore subquotients.
\item Consider  $\mathcal{E}_{2}\hookrightarrow \mathcal{E}_{1}\hookrightarrow \mathcal{F}\boxtimes\mathcal{G}$ (resp.  $\mathcal{E}_{2}'\hookrightarrow \mathcal{E}_{1}'\hookrightarrow \mathcal{F}'\boxtimes\mathcal{G}'$).  We have $\mathcal{E}_{2}\otimes\mathcal{E}_{2}'\hookrightarrow \mathcal{E}_{1}\otimes\mathcal{E}_{1}'\hookrightarrow (\mathcal{F}\otimes\mathcal{F}')\boxtimes(\mathcal{G}\otimes\mathcal{G}')$, and a quotient 
$\mathcal{E}_{1}\otimes\mathcal{E}_{1}'/\mathcal{E}_{2}\otimes\mathcal{E}_{2}' \twoheadrightarrow (\mathcal{E}_{1}/\mathcal{E}_{2})\otimes(\mathcal{E}_{1}'/\mathcal{E}_{2}')$. Therefore $ (\mathcal{E}_{1}/\mathcal{E}_{2})\otimes(\mathcal{E}_{1}'/\mathcal{E}_{2}')$ is of the desired form since objects of the desired form are closed under taking subquotients. 
\item Consider  $\mathcal{E}_{2}\hookrightarrow \mathcal{E}_{1}\hookrightarrow \mathcal{F}\boxtimes\mathcal{G}$. The dual of $\mathcal{E}_{1}/\mathcal{E}_{2}$ is a subobject of the dual $\mathcal{E}_{1}^{\vee}$ of $\mathcal{E}_{1}$, so it suffices to prove that $\mathcal{E}_{1}^{\vee}$ is of the desired form, which is clear since we have a quotient map $\mathcal{F}^{\vee}\boxtimes\mathcal{G}^{\vee}\twoheadrightarrow \mathcal{E}_{1}^{\vee}$. 
\end{enumerate}\par
Now we move to the second part of the lemma. We first prove that the morphism \eqref{structure} is faithfully flat. We consider the composition
\begin{equation}\label{flat}\pi_{1}^{diff}(X,x)\rightarrow \pi_{1}^{diff}(X\times_{k}Y,x\times y) \rightarrow \pi_{1}^{diff}(X,x)\times_{k}\pi_{1}^{diff}(Y,y), \end{equation}
where the first arrow is induced by the pull-back along $i_{y}$ and the second arrow is just \eqref{structure}.  Considering further composition with the projection from $\pi_{1}^{diff}(X,x)\times_{k}\pi_{1}^{diff}(Y,y)$ to $\pi_{1}^{diff}(X,x)$ (resp. $\pi_{1}^{diff}(Y,y))$, we obtain a morphism $\pi_{1}^{diff}(X,x)\rightarrow \pi_{1}^{diff}(X,x)$ (resp. $\pi_{1}^{diff}(X,x)\rightarrow \pi_{1}^{diff}(Y,y)$), induced by the pull-back along the identity $p_{X}\circ i_{y}$ (resp. the constant $p_{Y}\circ i_{y}$). This means that the image of the composition \eqref{flat} is exactly the subgroup $\pi_{1}^{diff}(X,x)\times\{1\}$ of $\pi_{1}^{diff}(X,x)\times_{k}\pi_{1}^{diff}(Y,y)$. Therefore $\pi_{1}^{diff}(X,x)\times\{1\}$ is contained in the image of the morphism \eqref{structure}. By the same reason $\{1\}\times \pi_{1}^{diff}(Y,y)$ is contained in the image of \eqref{structure}. We conclude the faithful flatness of \eqref{structure}.\par
By lemma \ref{subquotient} and Tannakian duality, any finite dimensional representation of $\pi_{1}^{diff}(X,x)\times\pi_{1}^{diff}(Y,y)$ is a subquotient of a representation of the form $\mathcal{F}_{x}\otimes\mathcal{G}_{y}$, for $\mathcal{F}\in\mathrm{Conn}(X)$ (resp. $\mathcal{G}\in\mathrm{Conn}(Y)$). Considered as a representation of $\pi_{1}^{diff}(X\times Y,x\times y)$,  $\mathcal{F}_{x}\otimes\mathcal{G}_{y}$ corresponds to the connection $\mathcal{F}\boxtimes\mathcal{G}$ by Tannakian duality. This proves the existence of the factorization \eqref{fac}. Since \eqref{structure} is faithfully flat, the category of representations of $\pi_{1}^{diff}(X,x)\times_{k} \pi_{1}^{diff}(Y,y)$, considered as a subcategory of the category of representations of $\pi_{1}^{diff}(X\times_{k}Y,x\times y)$, is closed under taking subquotients. Combining the first part of the lemma, we conclude that the second arrow in \eqref{fac} is an isomorphism.
\end{proof}
What will be used in the sequel is summarized as a corollary:
\begin{cor}
\label{cir}
The natural morphism 
\[\pi^{diff}_{1}(X\times_{k} Y,x\times y)\rightarrow \pi_{1}^{diff}(X,x)\times_{k}\pi_{1}^{diff}(Y,y)\] is faithfully flat. It is an isomorphism if and only if any object in $\mathrm{Conn}(X\times_{k} Y)$ is of the form $\mathcal{E}_{1}/\mathcal{E}_{2}$ for some $\mathcal{E}_{2}\hookrightarrow \mathcal{E}_{1}\hookrightarrow \mathcal{F}\boxtimes\mathcal{G}$, with $\mathcal{F}\in\mathrm{Conn}(X)$ and  $\mathcal{G}\in\mathrm{Conn}(Y)$. 
\end{cor}
\begin{proof}
This is clear from  Lemma \ref{tannakian_lemma}.
\end{proof}
\section{\normalsize Admissible formal decompositions}

Let $\bar{X}$ be smooth over $k$, say of dimension $n$, $D$ be a strict normal crossings divisor, and $X$ be $\bar{X}\setminus D$. Let $x\in D$ be a point through which $m$ irreducible components of $D$ pass. We can choose étale local coordinates $z_{1},...,z_{n}$ around $x$, such that the completed local ring $R=\hat{\mathcal{O}}_{X,x}$ at $x$ is isomorphic to $k[[z_{1},...,z_{n}]]$ and $D$ is defined by $z_{1}...z_{m}=0$. We write $S=R[\frac{1}{z_{1}},...,\frac{1}{z_{m}}]$ and differential operators $\partial_{i}=\partial _{z_{i}}$, $1\leq i\leq n$, which act on $R$ and $S$. 
\begin{defn}
A differential $S$-module is a finitely generated $S$-module $M$, together with actions by $\partial_{i}$, $1\leq i\leq n$, satisfying: 
\begin{enumerate}
\item $\partial_{i}(s\cdot m_{1}+m_{2})=\partial_{i}s\cdot m_{1}+s\cdot \partial_{i}m_{1}+\partial_{i}m_{2}$, for all $1\leq i\leq n$, $s\in S$ and $m_{1},m_{2}\in M$.
\item $[\partial_{i},\partial_{j}]=\partial_{i}\partial_{j}-\partial_{j}\partial_{i}$ acts by $0$, for all $1\leq i,j\leq n$.
\end{enumerate}\par
We remark that the differential $S$-modules form an abelian category and that the underlying $S$-module of a differential $S$-module is necessarily free of finite rank (Lemma 3.3.3 \cite{kedlaya2011good}). A differential $S$-module $\mathcal{M}$ is said to be regular, if there exists an $R$-lattice of $\mathcal{M}$ which is stable under the actions of  $z_{1}\partial_{1},...,z_{m}\partial_{m},\partial_{m+1},...,\partial_{n}$.
\end{defn}
Note $S$ is itself a differential $S$-module. Also note that regular differential $S$-modules are closed under taking subquotients, duals, and tensor products in the category of all differential $S$-modules.
\begin{emp}
\label{global}
Let $(\mathcal{E},\nabla)$ be an integrable connection on $X$. By restricting it along the natural morphism $\mathrm{Spec}\,S\rightarrow X$, we obtain a differential $S$-module, denoted by $(\mathcal{E},\nabla)\otimes_{\mathcal{O}_{X}}S$.
\end{emp}
\begin{emp}
\label{exponential}
Let $\phi$ be in $S$ and we define an exponential differential $S$-module $\mathcal{E}(\phi)$ as follows. The underlying $S$-module is $S$ itself,  with the  differential actions given by 
\[\partial_{i}\circ_{\phi}s=\partial_{i}s+s\cdot \partial_{i}\phi, 1\leq i\leq n, s\in S.\]
The exponential differential $S$-module $\mathcal{E}(\phi)$ is regular if and only if $\phi\in R$.
\end{emp}
\begin{emp}
\label{combine}
Let $\phi$ be a regular function on $X$. Consider the connection $(\mathcal{O}_{X},\nabla_{\phi})$, with
\[\nabla_{\phi}=\mathrm{d}+\mathrm{d}\phi.\]
The differential $S$-module $(\mathcal{O}_{X},\nabla_{\phi})\otimes_{\mathcal{O}_{X}}S$ from Example \ref{global} is the exponential one $\mathcal{E}(\phi)$ from Example \ref{exponential}.
\end{emp}
\begin{defn}[originally referred as ``bonne décomposition'' in 1.3 \cite{AIF_1993__43_5_1619_0}, Definition 3.4.6 \cite{kedlaya2011good}]
\label{admissible}
 An admissible formal decomposition of a differential $S$-module $\mathcal{M}$  is an isomorphism 
\[\mathcal{M}\cong \bigoplus_{i}\mathcal{E}(\phi_{i})\otimes_{S} \mathcal{R}_{i},\]
with exponential differential $S$-modules $\mathcal{E}(\phi_{i})$ and regular differential $S$-modules $\mathcal{R}_{i}$. Here tensor products are taken in the category of all differential $S$-modules.  
\end{defn}
The admissible formal decomposition is not unique. Indeed, in case $\phi_{i}\in R$, $\mathcal{E}(\phi_{i})$ is regular, and it can be combined with $\mathcal{R}_{i}$. Nevertheless, the collection of all $\phi_{i}+R$ in $S/R$ is independent of the chosen admissible formal decomposition (whenever it exists). See Remark 3.4.7 \cite{kedlaya2011good}. We use $\mathrm{Exp}(\mathcal{M})$ to denote the set of these $\phi_{i}+R$ in $S/R$. 

\begin{thm}[Hukuhara-Levelt-Turrittin, Theorem 3.54 \cite{van2012galois}, Theorem 2.3.3 \cite{10.1215/00127094-2010-041}]
\label{lt}
Let $X$ be of dimension $1$ and $x\in D$ be a boundary point of the smooth compactification $\bar{X}$ of $X$. Let $z$ be an étale local coordinate of $\bar{X}$ such that $z=0$ defines $x$. We have $R=k[[z]]$ and $S=R[\frac{1}{z}]$. For any differential $S$-module $\mathcal{M}$, there exists an integer $p>0$, such that the pullback of $\mathcal{M}$ along the $p$-power map $\mathrm{Spec}\,S \rightarrow \mathrm{Spec}\,S$, $z\mapsto z^{p}$ of the formal punctured disk, admits an admissible formal decompostion.
\end{thm}
In higher dimensions, it is not clear if an admissible formal decomposition exists. In fact we should ask for a ``good'' admissible formal decomposition, which is the main subject of \cite{mochizuki2009good}\cite{kedlaya2011good}. The upshot is that a ``good'' admissible formal decomposition exists after iterated blow-ups with centers supported in the singular locus of the connection, and passage to ramified coverings (noting that the passage to coverings already appears in Theorem \ref{lt}). For all differential $S$-modules appearing in this article, the existence of an admissible formal decompotion will be self-evident.
\begin{lem}
\label{exp}
Let $\mathcal{M}$ be a differential $S$-module such that an admissible formal decomposition exists. Then any subquotient $\mathcal{M}'$ of $\mathcal{M}$ has an admissible formal decomposition, and we have 
\[\mathrm{Exp}(\mathcal{M}')\subset \mathrm{Exp}(\mathcal{M}).\]
\end{lem}
\begin{proof}
We proceed by induction on the cardinality $l$ of $\mathrm{Exp}(\mathcal{M})$.  \par
Let the admissible formal decomposition of $\mathcal{M}$ be 
\[\mathcal{M} \cong \bigoplus_{1\leq i\leq l}\mathcal{E}(\phi_{i})\otimes_{S} \mathcal{R}_{i}.\]
We can assume that for $1\leq i\neq j\leq l$, we have $\phi_{i}-\phi_{j}\notin R$. \par
The case $l=1$ is clear, for $\mathcal{M'}\otimes _{S}\mathcal{E}(-\phi_{1})$ is a subquotient of $\mathcal{R}_{1}$ and so it is regular. Then $\mathcal{M}'$ is the tensor product of a regular differential $S$-module with $\mathcal{E}(\phi_{1})$.  \par
Now assume $l>1$. We prove the case  that $\mathcal{M}'$ is a subobject of $\mathcal{M}$. The cases of quotients and subquotients are similar.  We write
\[\mathcal{N}=\bigoplus_{1\leq i\leq l-1}\mathcal{E}(\phi_{i})\otimes_{S} \mathcal{R}_{i}.\]
Consider the projection $\mathcal{M}_{1}'$ (resp.$\mathcal{M}_{2}'$) of $\mathcal{M}'\hookrightarrow \mathcal{M}\cong \mathcal{N}\bigoplus (\mathcal{E}(\phi_{l})\otimes_{S}\mathcal{R}_{l})$ to $\mathcal{N}$ (resp.to $\mathcal{E}(\phi_{l})\otimes_{S} \mathcal{R}_{l}$). Then $\mathcal{M}_{1}'$ (resp.$\mathcal{M}_{2}'$) is a subobject of $\mathcal{N}$ (resp.$\mathcal{E}(\phi_{l})\otimes_{S} \mathcal{R}_{l}$), together with a quotient $\mathcal{M}'\twoheadrightarrow \mathcal{M}'_{1}$ (resp.$\mathcal{M}'\twoheadrightarrow \mathcal{M}'_{2}$). Note the admissible formal decomposition of $\mathcal{N}$ is clear by definition. Apply induction hypothesis to the subobject $\mathcal{M}_{1}'$ of $\mathcal{N}$, we see that an admissible formal decomposition of $\mathcal{M}_{1}'$ exists, with 
\[\mathrm{Exp}(\mathcal{M}'_{1})\subset \mathrm{Exp}(\mathcal{N})=\{\phi_{i}+R| 1\leq i\leq l-1\}.\]
We are done if the natural inclusion 
\[\mathcal{M}'\hookrightarrow \mathcal{M}'_{1}\bigoplus \mathcal{M}'_{2}\]
is an isomorphism. Let the cokernel be $\mathcal{L}$ and we prove $\mathcal{L}=0$.  By contradiction suppose $\mathcal{L}\neq 0$. Since the morphism $\mathcal{M}'\rightarrow \mathcal{M}'_{1}$ (resp.$\mathcal{M}'\rightarrow \mathcal{M}'_{2}$) is surjective, one sees that the composition $\mathcal{M}_{2}'\hookrightarrow \mathcal{M}_{1}'\bigoplus \mathcal{M}_{2}' \twoheadrightarrow \mathcal{L}$(resp. $\mathcal{M}_{1}'\hookrightarrow \mathcal{M}_{1}'\bigoplus \mathcal{M}_{2}' \twoheadrightarrow \mathcal{L}$) is surjective. Therefore $\mathcal{L}$ is a quotient of both $\mathcal{M}_{1}'$ and $\mathcal{M}_{2}'$. By induction hypothesis (applied to the quotient maps $\mathcal{M}'_{1}\twoheadrightarrow \mathcal{L}$ and $\mathcal{M}_{2}'\twoheadrightarrow\mathcal{L}$), an admissible formal decomposition of $\mathcal{L}$ exists and we have
\[\mathrm{Exp}(\mathcal{L})\subset \mathrm{Exp}(\mathcal{M}'_{1})\cap\mathrm{Exp}(\mathcal{M}'_{2})=\emptyset,\]
which is  a contradiction. 
\end{proof}
\section{\normalsize The counterexample}
\begin{thm}
\label{kun_diff}
Suppose $X=\mathbb{P}_{k}^{1}\setminus\{0\}=\mathrm{Spec}\,k[z^{-1}]$ and $Y=\mathbb{P}^{1}_{k}\setminus\{0\}=\mathrm{Spec}\,k[w^{-1}]$. Fix base points $x\in X(k)$ and $y\in Y(k)$. 
The natural morphism 
\[ \pi_{1}^{diff}(X \times_{k} Y, x\times y)\rightarrow \pi_{1}^{diff}(X,x)\times_{k}\pi_{1}^{diff}(Y,y)\]
is not an isomorphism.
\end{thm}
\begin{proof}

Consider the connection $(\mathcal{O}_{Z},{\nabla_{z^{-1}w^{-1}}})$ on $Z=X\times_{k} Y=\mathrm{Spec}k[z^{-1},w^{-1}]$, with
\[\nabla_{z^{-1}w^{-1}}=\mathrm{d}+\mathrm{d}(z^{-1}w^{-1}).\] By Corollary \ref{cir}, it suffices to show that $(\mathcal{O}_{Z},{\nabla_{z^{-1}w^{-1}}})$ is not a subquotient of some connection of the form $\mathcal{F}\boxtimes\mathcal{G}$, with $\mathcal{F}\in\mathrm{Conn}(X)$ and $\mathcal{G}\in\mathrm{Conn}(Y)$. Suppose this is not the case so there are such $\mathcal{F}$ and $\mathcal{G}$. Theorem \ref{lt} says that there exists an integer $p>0$ (resp. $q>0$), such that the pullback of $\mathcal{F}$ (resp.$\mathcal{G}$) along the ramified covering $X\rightarrow X$, $z^{-1}\mapsto z^{-p}$ (resp. $Y\rightarrow Y$, $w^{-1}\mapsto w^{-q}$) admits an admissible formal decomposition at $z=0$ (resp.$w=0$). We also pull back $(\mathcal{O}_{Z},{\nabla_{z^{-1}w^{-1}}})$ along the ramified covering $Z\rightarrow Z$, $z^{-1}\mapsto z^{-p}$, $w^{-1}\mapsto w^{-q}$, and the pullback is $(\mathcal{O}_{Z},{\nabla_{z^{-p}w^{-q}}})$. Replacing $\mathcal{F}$ (resp. $\mathcal{G}$) by its pullback to the ramified covering, we arrive at the following situation: $(\mathcal{O}_{Z},\nabla_{z^{-p}w^{-q}})$ is a subquotient of $\mathcal{F}\boxtimes\mathcal{G}$, such that $\mathcal{F}$ (resp. $\mathcal{G}$) admits an admissible formal decomposition 
\[\mathcal{F}\otimes_{\mathcal{O}_{X}}k((z))\cong\bigoplus_{1\leq i\leq l} \mathcal{E}(\phi_{i})\otimes_{k((z))}\mathcal{R}_{X,i}\] with $\phi_{i}\in k((z))$ at $z=0$ (resp.
\[\mathcal{G}\otimes_{\mathcal{O}_{Y}}k((w))\cong\bigoplus_{1\leq j\leq m} \mathcal{E}(\psi_{j})\otimes_{k((w))}\mathcal{R}_{Y,j}\] with $\psi_{j}\in k((w))$ at $w=0$). \par
Consider the compactification of $Z$ into $\mathbb{P}_{k}^{1}\times_{k}\mathbb{P}_{k}^{1}$, and we consider the admissible formal decompositions of $(\mathcal{O}_{Z},\nabla_{z^{-p}w^{-q}})$ and $\mathcal{F}\boxtimes\mathcal{G}$ at the boundary point defined by $z=w=0$. Now $R=k[[z,w]]$ and $S=R[\frac{1}{z},\frac{1}{w}]$. The connection $\mathcal{F}\boxtimes\mathcal{G}$ has an admissible formal decomposition 
\[(\mathcal{F}\boxtimes\mathcal{G}) \otimes_{\mathcal{O}_{Z}}S\cong \bigoplus_{1\leq i\leq l, 1\leq j\leq m} \mathcal{E}(\phi_{i}+\psi_{j})\otimes_{S} \mathcal{R}_{X,i}\otimes_{S}\mathcal{R}_{Y,j},\]
where with a slight abuse of notation we still use $\mathcal{R}_{X,i}$ (resp.$\mathcal{R}_{Y,j}$) to denote the regular differential $S$-module $\mathcal{R}_{X,i}\otimes_{k((z))}S$ (resp. $\mathcal{R}_{Y,j}\otimes_{k((w))}S$) obtained by extension of scalars from the regular differential $k((z))$-module $\mathcal{R}_{X,i}$ (resp.$k((w))$-module $\mathcal{R}_{Y,j}$).
An admissible formal decomposition of $(\mathcal{O}_{Z},\nabla_{z^{-p}w^{-q}})\otimes_{\mathcal{O}_{Z}}S$ is self-evident, as it is $\mathcal{E}(z^{-p}w^{-q})$ by Example $\ref{combine}$.  By assumption, $(\mathcal{O}_{Z},{\nabla_{z^{-p}w^{-q}}})$ is a subquotient of $\mathcal{F}\boxtimes\mathcal{G}$, so by Lemma \ref{exp} there exist some $i$ and $j$ such that 
\[z^{-p}w^{-q}-\phi_{i}-\psi_{j}\in R=k[[z,w]].\]
This is a contradiction since $z^{-p}w^{-q}$ can not be written as the summation of $\phi_{i}\in k((z))$, $\psi_{j}\in k((w))$, and an element in $k[[z,w]]$. 
\end{proof}
\section*{\normalsize Acknowledgement}
The author would like to thank the anonymous reviewer for many useful comments. 
 \bibliographystyle{plain}
 \bibliography{D-module}

\end{document}